\documentclass[12pt,reqno]{article}

\usepackage[usenames]{color} 
\usepackage{amssymb} 
\usepackage{amsmath} 
\usepackage{amsthm} 
\usepackage{amsfonts} 
\usepackage{amscd} 
\usepackage{graphicx} 
\usepackage{xcolor} 
\usepackage{float}
\usepackage{array}
\usepackage{booktabs}

\usepackage[hidelinks]{hyperref}

\definecolor{webgreen}{rgb}{0,.5,0} 
\definecolor{webbrown}{rgb}{.6,0,0}

\usepackage{fullpage}

\usepackage{graphics} 
\usepackage{latexsym} 
\usepackage{epsf} 

\newcommand{\seqnum}[1]{\href{https://oeis.org/#1}{\rm \underline{#1}}}

\begin{document}


\theoremstyle{plain} 
\newtheorem{theorem}{Theorem} 
\newtheorem{corollary}[theorem]{Corollary} 
\newtheorem{lemma}[theorem]{Lemma} 
\newtheorem{proposition}[theorem]{Proposition}

\theoremstyle{definition} 
\newtheorem{definition}[theorem]{Definition} 
\newtheorem{example}[theorem]{Example} 
\newtheorem{conjecture}[theorem]{Conjecture}

\theoremstyle{remark} 
\newtheorem{remark}[theorem]{Remark}

\title{Repetition Avoidance in Curling-Number Transforms}
\date{}

\author{Geoffrey Caveney\\
Jersey City, NJ\\
USA\\
\href{mailto:geoffreycaveney@gmail.com}{\tt geoffreycaveney@gmail.com}
\and
Haoxuan (Jason) Dong\\
Melbourne, Victoria\\
Australia\\
\href{mailto:donghaoxuan13818851792@gmail.com}{\tt donghaoxuan13818851792@gmail.com}
\and
Jeffrey Shallit\\
School of Computer Science\\
University of Waterloo\\
Waterloo, ON N2L 3G1\\
Canada\\
\href{mailto:shallit@uwaterloo.ca}{\tt shallit@uwaterloo.ca}}

\maketitle

\vskip .2 in
\begin{abstract}
We study repetition avoidance in a word ${\bf w}$ and its curling-number transform $C({\bf w})$.  For alphabets of sizes $2$, $3$, and $4$, we use Thue--Morse-based morphic constructions and exhaustive finite searches.  A ternary word for which both ${\bf w}$ and $C({\bf w})$ are overlap-free has length at most $84$, whereas over four letters an infinite example exists.  Hence $4$ is the smallest alphabet size admitting simultaneous infinite overlap-freeness.  The infinite constructions are verified in {\tt Walnut}; the finite maxima are obtained by exhaustive breadth-first search and checked independently.
\end{abstract}

\section{Introduction}

The curling number of a finite word measures the maximum degree of repetition of its suffixes.  Recording the curling number of each prefix gives the curling-number transform.  We study simultaneous repetition avoidance in a word and its curling-number transform.

Repetition avoidance is a classical topic in combinatorics on words, beginning with Thue's constructions of infinite words avoiding prescribed repetitions~\cite{Thue:1912}.  Curling numbers arose in a different setting, from repetitions at the ends of finite sequences~\cite{Bult,Chaffin}.  The original work also considered curling-number transforms of familiar sequences~\cite{Bult}; here we impose repetition-avoidance conditions simultaneously on a word and on its transform.

Our main result gives the exact alphabet threshold for simultaneous overlap-freeness.  Over three letters the maximum length is $84$ (Theorem~\ref{thm:ternary-overlap-max}); over four letters an infinite example exists (Theorem~\ref{thm:four-overlap}).  Hence the threshold is $4$.

We also obtain four boundary cases over binary and ternary alphabets.  At the exponents $3$, $5/2$, $9/4$, and $7/3$, the relevant $\alpha$-free case has a finite maximum, whereas the corresponding $\alpha^+$-free case admits an infinite construction.  Table~\ref{tab:results} summarizes these results together with the exact alphabet threshold.  The notation $\alpha$-free and $\alpha^+$-free is recalled in Section~\ref{sec:prelim}.

\begin{table}[ht]
\centering
\small
\begin{tabular}{@{}c>{\raggedright\arraybackslash}p{0.26\textwidth}>{\raggedright\arraybackslash}p{0.26\textwidth}>{\raggedright\arraybackslash}p{0.22\textwidth}@{}}
\toprule
Alphabet size & Condition on the source word & Condition on the transform & Outcome \\
\midrule
$2$ & cubefree & cubefree & maximum length $17$ \\
$2$ & $3^+$-free & overlap-free ($2^+$-free) & infinite \\
$2$ & cubefree & overlap-free ($2^+$-free) & maximum length $13$ \\
$2$ & $(5/2)^+$-free & $3^+$-free & infinite \\
$2$ & $(5/2)$-free & $3^+$-free & maximum length $75$ \\
\midrule
$3$ & overlap-free ($2^+$-free) & overlap-free ($2^+$-free) & maximum length $84$ \\
$3$ & $(9/4)^+$-free & overlap-free ($2^+$-free) & infinite \\
$3$ & $(9/4)$-free & overlap-free ($2^+$-free) & maximum length $84$ \\
$3$ & overlap-free ($2^+$-free) & $(7/3)^+$-free & infinite \\
$3$ & overlap-free ($2^+$-free) & $(7/3)$-free & maximum length $84$ \\
\midrule
$4$ & overlap-free ($2^+$-free) & overlap-free ($2^+$-free) & infinite \\
\bottomrule
\end{tabular}
\caption{Summary of the repetition-avoidance results proved in this paper. }
\label{tab:results}
\end{table}

The proofs use two methods.  Infinite examples come from uniform morphisms applied to a $2$-block coding of the Thue--Morse word, with the avoidance properties verified in {\tt Walnut}.  Exact finite bounds come from exhaustive breadth-first search over valid prefixes.  Section~\ref{sec:prelim} introduces the notation and these two methods.  Sections~\ref{sec:binary}, \ref{sec:ternary}, and \ref{sec:four} treat alphabets of sizes $2$, $3$, and $4$, respectively.  The {\tt Walnut} scripts appear in Appendix~\ref{app:walnut}, and Appendix~\ref{sec:repro} describes the standalone finite-search verification.

\section{Preliminaries and proof methods}
\label{sec:prelim}

\subsection{Repetitions and curling-number transforms}

Let $x=x[0..n-1]$ be a finite word of length $n\geq 1$.  If $x=y^k$ for a nonempty word $y$ and an integer $k\geq 1$, then $x$ is a $k$-power.  A $2$-power is a {\it square} (for example, {\tt murmur}), and a $3$-power is a {\it cube} (for example, {\tt shshsh}).  An {\it overlap} is a word of the form $axaxa$, where $a$ is a single letter and $x$ is possibly empty; {\tt alfalfa} is an example.

A finite word is squarefree, cubefree, or overlap-free if it contains no factor of the corresponding type.  We use the same terminology for infinite words.  The Thue--Morse word ${\bf t}=01101001\cdots$, the fixed point beginning in $0$ of the morphism $0\mapsto01$, $1\mapsto10$, is overlap-free~\cite{Thue:1912}.

The curling number of a finite word $x$ is the largest integer $k$ for which $x=yz^k$ with $z$ nonempty.  For example, {\tt brouhaha} has curling number $2$~\cite{Bult,Chaffin}.

For a finite or infinite word $x$, define $C(x)$ by letting $C(x)[i]$ be the curling number of the prefix $x[0..i]$.  We call $C(x)$ the {\it curling-number transform} of $x$.  This direct-prefix convention differs from that of van de Bult et al.~\cite{Bult} only in that their initial $1$ is omitted.  If $x$ is $k$-power-free for an integer $k$, then $C(x)$ takes values in $\{1,2,\ldots,k-1\}$.

\begin{remark}
\label{rem:cubefree-curl}
Let $x$ be a finite or infinite cubefree word. For every position $n$ in the domain of $x$,
\[
C(x)[n]\in\{1,2\},
\]
and
\[
C(x)[n]=2
\quad\Longleftrightarrow\quad
x[0..n]\text{ ends in a square}.
\]
Indeed, cubefreeness rules out curling number at least $3$, while curling number at least $2$ is equivalent to the existence of a square suffix.
\end{remark}

For a squarefree infinite word, the transform is $111\cdots$, so we restrict attention to words containing squares.

We also use fractional powers.  A finite word $x=x[0..n-1]$ has period $p$, where $1\leq p\leq n$, if $x[i]=x[i+p]$ for $0\leq i<n-p$.  If $p$ is the least period of $x$, then the exponent of $x$ is $n/p$, written $\exp(x)=n/p$; for example, {\tt entente} has exponent $7/3$.  A finite or infinite word is $\alpha$-power-free (or $\alpha$-free) if it has no nonempty factor of exponent at least $\alpha$.  It is $\alpha^+$-power-free (or $\alpha^+$-free) if it has no factor of exponent strictly greater than $\alpha$.  Thus $2^+$-free is equivalent to overlap-free.

\subsection{The Thue--Morse source, morphic constructions, and Walnut}
\label{sec:morphic-walnut}

All of the infinite constructions use the same auxiliary sequence.  Starting from the Thue--Morse word ${\bf t}=01101001\cdots$, define ${\bf p}=1321201\cdots$ over the alphabet $\{0,1,2,3\}$ by 
$$ {\bf p}[i] = \begin{cases} 
0, & \text{if } {\bf t}[i..i+1] = 00; \\
1, & \text{if } {\bf t}[i..i+1] = 01; \\
2, & \text{if } {\bf t}[i..i+1] = 10; \\
3, & \text{if } {\bf t}[i..i+1] = 11. 
\end{cases} 
$$ 
The word ${\bf p}$ is $2$-automatic and is generated by a $4$-state automaton.  It is sequence \seqnum{A245188} in the {\it On-Line Encyclopedia of Integer Sequences} (OEIS)~\cite{oeis}.  We search for uniform morphisms $h$ for which $h({\bf p})$, or a slight modification of it, satisfies the required pair of avoidance conditions.  Candidate images are generated by breadth-first search, extending each image from the outside inward.

Each candidate is then verified in {\tt Walnut}, an automatic theorem prover for automatic sequences~\cite{Mousavi:2016,Shallit:2024}.
The automaton $\tt P$ for $\bf p$ can be constructed in {\tt Walnut} as follows: 
\begin{verbatim} 
def p00 "T[n]=@0 & T[n+1]=@0": 
def p01 "T[n]=@0 & T[n+1]=@1": 
def p10 "T[n]=@1 & T[n+1]=@0": 
def p11 "T[n]=@1 & T[n+1]=@1": 
combine P p00=0 p01=1 p10=2 p11=3: 
\end{verbatim}

Several {\tt Walnut} predicates below rule out fractional powers using a period $p$ that need not be least.  This is sufficient.  For fixed $p$ and $\alpha>1$, the shortest length $L$ satisfying $L/p>\alpha$ is
\[
L=\lfloor \alpha p\rfloor+1.
\]
It therefore suffices to compare positions separated by $p$ throughout a factor of this length.  If the chosen $p$ is not least, the least period $q$ is at most $p$, so the exponent can only increase.  Conversely, a factor with least period $q$ and exponent greater than $\alpha$ is detected by its prefix of length $\lfloor\alpha q\rfloor+1$.  For exponent at least $\alpha$, the corresponding length is $\lceil\alpha p\rceil$.  This accounts for the bounds ${\tt t<=2*n}$, ${\tt 2*t<=3*n}$, ${\tt 4*t<=5*n}$, and ${\tt 3*t<=4*n}$ used in the verification scripts.

\subsection{Exhaustive prefix search}

The finite nonexistence results are obtained by exhaustive breadth-first search on valid prefixes.  At level $n$ we retain every length-$n$ word satisfying the two required avoidance conditions.  To construct level $n+1$, every alphabet symbol is appended to every retained word, and the resulting extensions are tested.

This search is exhaustive because validity is prefix-hereditary.  Once a forbidden factor appears in the source word, it remains present in every extension; the same is true for the transform because $C(u)$ is a prefix of $C(uv)$.  Hence every valid word of length $n+1$ extends a word retained at level $n$.  For each exact finite bound we give the size of the last nonempty level and of the next, empty level.

\section{The binary alphabet}
\label{sec:binary}

For binary words, simultaneous cubefreeness is finite, but two relaxations admit infinite constructions.  We also determine the boundary case for each relaxation.

\begin{theorem} 
\label{thm:binary-cube-cube-max}
There is no cubefree binary word of length $> 17$ whose curling-number transform is 
also cubefree. 
\end{theorem}

\begin{proof} 
The exhaustive search finds exactly $2$ valid words of length $17$ and none of length $18$.  Thus the maximum length is $17$.  Up to interchange of the two binary symbols, the unique word of length $17$ is
$$00110100110010110,$$
with curling-number transform $12121122121221221$.
\end{proof}

\subsection{An overlap-free transform}

If cubes are allowed but powers of exponent greater than $3$ are forbidden, an infinite example becomes possible.

\begin{theorem} 
\label{thm:binary-threeplus-overlap}
There is an infinite $3^+$-free binary word whose curling-number transform is overlap-free. 
\end{theorem}

\begin{proof} 
Define the uniform morphism
\begin{align*} 
h_1(0) &= 01110 \\
h_1(1) &= 00101 \\
h_1(2) &= 10111 \\
h_1(3) &= 01100 . 
\end{align*}
Let $Q1=h_1({\bf p})$.  The {\tt Walnut} verification script is given in Appendix~\ref{app:h1}.  The predicates ${\tt has2}$ and ${\tt has3}$ detect square and cube suffixes of the prefix ending at position $n$.  The predicate ${\tt threeplusfree}$ verifies that $Q1$ has no factor of exponent greater than $3$, so every curling number belongs to $\{1,2,3\}$.  The mutually exclusive predicates ${\tt cn1}$, ${\tt cn2}$, and ${\tt cn3}$ therefore define $C(Q1)$ exactly.  Finally, ${\tt checkoverlap}$ verifies that $C(Q1)$ is overlap-free.
\end{proof}

Requiring the source word itself to be cubefree makes the problem finite again.

\begin{theorem} 
\label{thm:binary-cube-overlap-max}
The longest cubefree binary word whose curling-number transform is overlap-free has length $13$. 
\end{theorem}

\begin{proof} 
The exhaustive search finds exactly $6$ valid words of length $13$ and none of length $14$, so the maximum length is $13$.  Up to interchange of the two symbols, the three length-$13$ words beginning with $0$ are
$$\{0110100110100, 0110100110101, 0110100110110\}.$$
Each has curling-number transform $1121122121122$.
\end{proof}

\subsection{\texorpdfstring{A $3^+$-free transform}{A 3+-free transform}}

We keep the $3^+$-free condition on the transform and lower the allowed source exponent.

\begin{theorem} 
\label{thm:binary-52plus-3plus}
There is an infinite $(5/2)^+$-free binary word whose curling-number transform is $3^+$-free. 
\end{theorem}

\begin{proof} 
Let
\begin{align*} 
h_2(0) &= 1011001001101001011001101 \\
h_2(1) &= 0110010011010010110010100 \\
h_2(2) &= 1101100101101001100101001 \\
h_2(3) &= 1101100100110100101100100 .
\end{align*}
Let $Q2=h_2({\bf p})$.  Its {\tt Walnut} verification script is given in Appendix~\ref{app:h2}.

The computation of ${\tt plus52free}$ reaches a largest intermediate automaton of $1{,}467{,}844$ states.  The predicate ${\tt cn2q2}$ detects square suffixes.  Since ${\tt plus52free}$ verifies that $Q2$ is $(5/2)^+$-free, $Q2$ is in particular cubefree; Remark~\ref{rem:cubefree-curl} then identifies ${\tt D2}$ with $C(Q2)$.  Finally, ${\tt check2}$ verifies that ${\tt D2}=C(Q2)$ is $3^+$-free.
\end{proof}

\begin{theorem} 
\label{thm:binary-52-3plus-max}
The longest $(5/2)$-free binary word whose curling-number transform is $3^+$-free has length $75$. 
\end{theorem}

\begin{proof} 
The exhaustive search finds exactly $4$ valid words of length $75$ and none of length $76$, hence the maximum length is $75$.  The two length-$75$ words beginning with $0$ are $w0$ and $w1$, where
$$w = 01100110110010011010010110010011011001011010011011001001101001011001101001.$$
Thus $w0$ and $w1$ are obtained by appending $0$ and $1$, respectively, to the common length-$74$ prefix $w$.  Together with their binary complements, these are all four words at level $75$.  Both $w0$ and $w1$ have curling-number transform
$$112121221222112221122122212112221122211221222121122211222112212221212212212.$$
\end{proof}

\section{The ternary alphabet}
\label{sec:ternary}

For ternary words, simultaneous overlap-freeness holds only up to length $84$.  The same bound is obtained for either of the two one-sided non-strict relaxations considered below.  We next relax the source and transform conditions separately.

\subsection{The simultaneous overlap-free obstruction}

\begin{theorem} 
The longest overlap-free ternary word whose curling-number transform is overlap-free has length $84$.
\label{thm:ternary-overlap-max}
\end{theorem}

\begin{proof} 
The exhaustive search finds $6048$ valid words of length $84$ and none of length $85$, so the maximum length is $84$.  One length-$84$ example, grouped for readability, is
\begin{center}
\small\ttfamily
01101001020020\;21211002002100\;21101001020020\\
21101102002021\;21100200210021\;12201122021100
\end{center}
Its curling-number transform is
\begin{center}
\small\ttfamily
11211221211212\;21122121122112\;12211221211212\\
21211221121221\;12212112211212\;21211212211212
\end{center}
and is also overlap-free.
\end{proof}

\subsection{Relaxing the source exponent}

\begin{theorem} 
\label{thm:ternary-94plus-overlap}
There is an infinite $(9/4)^+$-free ternary word whose curling-number transform is overlap-free. 
\end{theorem}

\begin{proof} 
Define
\begin{align*} 
h_3(0) &= 00100120012011201200101101211212 \\
h_3(1) &= 00100120012011201200120022122102 \\
h_3(2) &= 21002022021221210020021002112122 \\
h_3(3) &= 21002022001220010020021002102210 .
\end{align*}
Let $B3=h_3({\bf p})$.  Its {\tt Walnut} verification script is given in Appendix~\ref{app:h3}.

The predicate ${\tt free94}$ verifies that $B3$ is $(9/4)^+$-free, hence cubefree.  By Remark~\ref{rem:cubefree-curl}, ${\tt D3}=C(B3)$.  With the $0$-based indexing used here, ${\tt check3}$ verifies that
\[
C(h_3({\bf p}))[n]=1+{\bf t}[n+3]\qquad(n\ge 0).
\]
Thus $C(h_3({\bf p}))$ is obtained from the suffix ${\bf t}[3..]$ by the bijective coding $0\mapsto1$, $1\mapsto2$, and is therefore overlap-free.
\end{proof}

\begin{theorem} 
\label{thm:ternary-94-overlap-max}
The longest $(9/4)$-free ternary word whose curling-number transform is overlap-free has length $84$. 
\end{theorem}

\begin{proof} 
The exhaustive search finds $10368$ valid words of length $84$ and none of length $85$, so the maximum length is $84$.  The example displayed in the proof of Theorem~\ref{thm:ternary-overlap-max} is among the words of length $84$.
\end{proof}

\subsection{Relaxing the transform exponent}

\begin{theorem} 
\label{thm:ternary-overlap-73plus}
There is an infinite overlap-free ternary word whose curling-number transform is $(7/3)^+$-free. 
\end{theorem}

\begin{proof} 
Define
\begin{align*} 
h_4(0) &= 001001200122 \\
h_4(1) &= 001011010200 \\
h_4(2) &= 101100211002 \\
h_4(3) &= 100120012200 .
\end{align*}
Let $B4=h_4({\bf p})$.  The {\tt Walnut} verification script is given in Appendix~\ref{app:h4}.  The predicate ${\tt overlapfree4}$ verifies that $B4$ is overlap-free, hence cubefree.  Remark~\ref{rem:cubefree-curl} then shows that the square-suffix detector used to define ${\tt D4}$ gives exactly $C(B4)$.  Finally, ${\tt free73}$ verifies that this transform is $(7/3)^+$-free.
\end{proof}

\begin{theorem} 
\label{thm:ternary-overlap-73-max}
The longest overlap-free ternary word whose curling-number transform is $(7/3)$-free has length $84$. 
\end{theorem}

\begin{proof} 
Here the exhaustive search again finds $6048$ valid words of length $84$ and none of length $85$.  Thus the maximum length is $84$, and the example displayed in Theorem~\ref{thm:ternary-overlap-max} is one of the length-$84$ words.
\end{proof}

\section{Four letters: simultaneous overlap-freeness}
\label{sec:four}

\begin{theorem} 
\label{thm:four-overlap}
There is an infinite overlap-free word over the alphabet $\{0,1,2,3\}$ whose curling-number transform is overlap-free. 
\end{theorem}

\begin{proof} 
Define $h_5$ by
\begin{align*} 
h_5(0) &= 1001200122322300\\
h_5(1) &= 1001200122003220\\
h_5(2) &= 0313110021100200\\
h_5(3) &= 0313112202203003 .
\end{align*}
Consider the word
\[
w=00h_5({\bf p})=0010012001220032200313112202203\cdots.
\]
Its curling-number transform begins $12112211212212112\cdots$.  In the {\tt Walnut} script, ${\tt QP5}$ represents $w$; the verification script is given in Appendix~\ref{app:h5}.  The predicate ${\tt overlapfree5}$ verifies that $w$ is overlap-free, hence cubefree.  Remark~\ref{rem:cubefree-curl} therefore identifies ${\tt D5}$ with $C(w)$, and ${\tt check5}$ verifies
\[
C(w)[n]=1+{\bf t}[n+3]\qquad(n\ge 0).
\]
Hence $C(w)$ is the bijective recoding $0\mapsto1$, $1\mapsto2$ of ${\bf t}[3..]$, and is overlap-free.
\end{proof}

\begin{corollary}
\label{cor:alphabet-threshold}
The smallest alphabet on which there exists an infinite word $w$ such that both $w$ and $C(w)$ are overlap-free has size $4$.
\end{corollary}

\begin{proof}
Theorem~\ref{thm:ternary-overlap-max} rules out such an infinite word over three letters, while Theorem~\ref{thm:four-overlap} gives an infinite example over four letters.
\end{proof}

\begin{remark}
The constructions in Theorems~\ref{thm:ternary-94plus-overlap} and~\ref{thm:four-overlap} start from the same $2$-block coding ${\bf p}$ and have the same curling-number transform:
\[
C(h_3({\bf p}))[n]=C(w)[n]=1+{\bf t}[n+3]\qquad(n\ge 0).
\]
Thus both transforms are bijective codings of the same shifted Thue--Morse word.  A general characterization of morphic words with this property is not known.
\end{remark}

\appendix
\section{Walnut verification scripts for the infinite constructions}
\label{app:walnut}

The following {\tt Walnut} scripts certify the five infinite constructions and are reproduced without alteration.

\subsection{\texorpdfstring{Binary $3^+$-free source and overlap-free transform}{Binary 3+-free source and overlap-free transform}}
\label{app:h1}
\begin{verbatim} 
morphism h1 "0->01110 1->00101 2->10111 3->01100": 
image Q1 h1 P:

eval threeplusfree "~Ei,n n>=1 & At (t<=2*n) => Q1[i+t]=Q1[i+n+t]":

def has2 "Ei,c c>=1 & i+2*c=n+1 & At (t<c) => Q1[i+t]=Q1[i+t+c]":
def has3 "Ei,c c>=1 & i+3*c=n+1 & At (t<2*c) => Q1[i+t]=Q1[i+t+c]":
def cn1 "~$has2(n)":
def cn2 "$has2(n) & ~$has3(n)":
def cn3 "$has3(n)":
combine D1 cn1=1 cn2=2 cn3=3:

eval checkoverlap "~Ei,n n>=1 & At (t<=n) => D1[i+t]=D1[i+t+n]": 
\end{verbatim}

\subsection{\texorpdfstring{Binary $(5/2)^+$-free source and $3^+$-free transform}{Binary (5/2)+-free source and 3+-free transform}}
\label{app:h2}
\begin{verbatim} 
morphism h2 "0->1011001001101001011001101 1->0110010011010010110010100 
2->1101100101101001100101001 3->1101100100110100101100100": 
image Q2 h2 P: 
eval plus52free "~Ei,n n>=1 & At (2*t<=3*n) => Q2[i+t]=Q2[i+n+t]"::
def cn2q2 "Ei,c c>=1 & i+2*c=n+1 & At (t<c) => Q2[i+t]=Q2[i+t+c]":
def cn1q2 "~$cn2q2(n)":
combine D2 cn1q2=1 cn2q2=2:

eval check2 "~Ei,n n>=1 & At (t<=2*n) => D2[i+t]=D2[i+t+n]":: 
\end{verbatim}

\subsection{\texorpdfstring{Ternary $(9/4)^+$-free source and overlap-free transform}{Ternary (9/4)+-free source and overlap-free transform}}
\label{app:h3}
\begin{verbatim} 
morphism h3 
"0->00100120012011201200101101211212  1->00100120012011201200120022122102 
2->21002022021221210020021002112122 3->21002022001220010020021002102210":

image B3 h3 P:

eval free94 "~Ei,n n>=1 & At (4*t<=5*n) => B3[i+t]=B3[i+n+t]"::

def curl3sq "Ei,c c>=1 & i+2*c=n+1 & At (t<c) => B3[i+t]=B3[i+t+c]":
def nocurl3 "~$curl3sq(n)":
combine D3 curl3sq=2 nocurl3=1:

eval check3 "An (T[n+3]=@0 <=> D3[n]=@1)": 
\end{verbatim}

\subsection{\texorpdfstring{Ternary overlap-free source and $(7/3)^+$-free transform}{Ternary overlap-free source and (7/3)+-free transform}}
\label{app:h4}
\begin{small}
\begin{verbatim} 
morphism h4 "0->001001200122 1->001011010200 2->101100211002 3->100120012200":

image B4 h4 P:

eval overlapfree4 "~Ei,n n>=1 & At (t<=n) => B4[i+t]=B4[i+n+t]"::

def curl4sq "Ei,c c>=1 & i+2*c=n+1 & At (t<c) => B4[i+t]=B4[i+t+c]":
def nocurl4 "~$curl4sq(n)":
combine D4 curl4sq=2 nocurl4=1:

eval free73 "~Ei,n n>=1 & At (3*t<=4*n) => D4[i+t]=D4[i+n+t]": 
\end{verbatim}
\end{small}

\subsection{Four-letter simultaneous overlap-freeness}
\label{app:h5}
\begin{verbatim} 
morphism h5 "0->1001200122322300 1->1001200122003220 
2->0313110021100200 3->0313112202203003":

image Q5 h5 P:

def qp0 "(n<=1) | (n>1 & Q5[n-2]=@0)": 
def qp1 "n>1 & Q5[n-2]=@1": 
def qp2 "n>1 & Q5[n-2]=@2": 
def qp3 "n>1 & Q5[n-2]=@3":

combine QP5 qp0=0 qp1=1 qp2=2 qp3=3:

eval overlapfree5 "~Ei,n n>=1 & At (t<=n) => QP5[i+t]=QP5[i+n+t]":

def curl5sq "Ei,c c>=1 & i+2*c=n+1 & At (t<c) => QP5[i+t]=QP5[i+t+c]":
def nocurl5 "~$curl5sq(n)":
combine D5 curl5sq=2 nocurl5=1:

eval check5 "An (D5[n]=@1 <=> T[n+3]=@0)": 
\end{verbatim}

\section{Computational reproducibility}
\label{sec:repro}

The supplementary C++17 program {\tt bfs\_verify.cpp} reproduces the six finite maximality results using the exhaustive prefix search from Section~\ref{sec:prelim}.  In all six cases the source condition implies cubefreeness, so Remark~\ref{rem:cubefree-curl} allows the next transform symbol to be obtained from a square-suffix test.

The program also implements the general definition of curling number, without assuming cubefreeness, as a separate check.  It verifies that every displayed example occurs at the stated terminal level and has the stated transform.  The file {\tt bfs\_verify\_output.txt} records the six terminal counts and these independent checks.  The program can be compiled and run with
\begin{verbatim}
c++ -O3 -std=c++17 bfs_verify.cpp -o bfs_verify
./bfs_verify
\end{verbatim}
The supplementary archive also contains the consolidated {\tt Walnut} input file \path{walnut_final.txt}, a README describing the reproducibility files, and SHA-256 checksums for the archived components.


\section*{Declaration of AI usage}

GPT-5.6 Sol and GPT-5.6 Sol Pro, including runs using Ultra mode, assisted in the search for the positive constructions in Sections~\ref{sec:binary}--\ref{sec:four}, in reviewing and debugging {\tt Walnut} predicates and scripts, in editing and checking proof explanations, and in preparing the supplementary C++ reproducibility verifier.  The authors retain responsibility for the final {\tt Walnut} executions, verification of the Walnut scripts and computational outputs, and all mathematical claims.  Two authors carried out the finite searches independently in APL and Python; both implementations and the supplementary C++ verifier give the reported terminal counts.

\end{document}